\documentclass[a4paper,11pt]{amsart}
\usepackage{latexsym,bm}
\usepackage{mathrsfs,amsmath,amssymb,amsthm,enumerate}
\usepackage[color,arrow,matrix]{xy}
\usepackage{color}
\usepackage{tikz}
\usepackage{graphicx}

\theoremstyle{plain}
\newtheorem{theorem}{Theorem}[section]
\newtheorem{proposition}[theorem]{Proposition}
\newtheorem{lemma}[theorem]{Lemma}
\newtheorem{corollary}[theorem]{Corollary}
\numberwithin{equation}{section}

\theoremstyle{definition}
\newtheorem{definition}[theorem]{Definition}
\newtheorem{remark}[theorem]{Remark}
\newtheorem{example}[theorem]{Example}
\newtheorem{question}[theorem]{Question}
\newtheorem{conjecture}[theorem]{Conjecture}

\newcommand{\C}{\mathbb{C}}

\newcommand{\Z}{\mathbb{Z}}

\newcommand{\CP}{\mathbb{C}P}

\renewcommand{\a}{\mathbf{a}}
\renewcommand{\r}{\mathbf{r}}
\newcommand{\s}{\mathbf{s}}
\newcommand{\h}{\mathbf{h}}
\renewcommand{\1}{\mathbf{1}}
\newcommand{\balpha}{\boldsymbol{\alpha}}
\newcommand{\bbeta}{\boldsymbol{\beta}}

\newcommand{\Pic}{\operatorname{Pic}}

\DeclareMathOperator{\wed}{wed}
\DeclareMathOperator{\link}{Lk}

\DeclareMathOperator{\proj}{Proj}

\def\red#1{{\textcolor{red}{#1}}}
\def\blue#1{\textcolor{blue}{#1}}

    \tikzstyle{v}=[circle, draw, solid, fill, inner sep=0pt, minimum width=4pt]

\begin{document}
\title[Wedge operations and torus symmetries II]{Wedge operations and torus symmetries II}

\author[S.Choi]{Suyoung Choi}
\address{Department of Mathematics, Ajou University, 206, World cup-ro, Yeongtong-gu, Suwon, 443-749, Republic of Korea}
\email{schoi@ajou.ac.kr}

\author[H.Park]{Hanchul Park}
\address{School of Mathematics, Korea Institute for Advanced Study (KIAS), 85 Hoegiro Dongdaemun-gu, Seoul 130-722, Republic of Korea}
\email{hpark@kias.re.kr}

\thanks{The first author was supported by Basic Science Research Program through the National Research Foundation of Korea(NRF) funded by the Ministry of Science, ICT \& Future Planning(NRF-2012R1A1A2044990).}

\date{\today}

%14M25   	Toric varieties, Newton polyhedra
%57S25   	Groups acting on specific manifolds
%52B11     $n$-dimensional polytopes
%13F55     Face and Stanley-Reisner rings; simplicial complexes
%18A10     Graphs, diagram schemes, precategories


%\subjclass[2010]{primary 57N65; secondary 57S17, 05E45}
%\keywords{toric manifold, quasitoric manifold, small cover}

\begin{abstract}
    A fundamental idea in toric topology is that classes of manifolds with well-behaved torus actions (simply, toric spaces) are classified by pairs of simplicial complexes and (non-singular) characteristic maps. {The authors in their previous paper provided} a new way to find all characteristic maps on a simplicial complex $K(J)$ obtainable by a sequence of wedgings from $K$. The main idea was that characteristic maps on $K$ theoretically determine all possible characteristic maps on a wedge of $K$.

    In this work, we further develop our previous work for classification of toric spaces. For a star-shaped simplicial sphere $K$ of dimension $n-1$ with $m$ vertices, the Picard number $\Pic(K)$ of $K$ is $m-n$. We refer to $K$ a \emph{seed} if $K$ cannot be obtained by wedgings. First, we show that, for a fixed positive integer $\ell$, there are at most finitely many seeds of Picard number $\ell$ supporting characteristic maps. As a corollary, the conjecture proposed by  V.~V.~Batyrev in 1991 is solved affirmatively.

    Second, we investigate a systematic method to find all characteristic maps on $K(J)$ using combinatorial objects called (realizable) \emph{puzzles} that only depend on a seed $K$.
    These two facts lead to a practical way to classify the toric spaces of fixed Picard number.
\end{abstract}
\maketitle

\tableofcontents
\section{Introduction}
    A \emph{toric variety} of dimension $n$ is a normal algebraic variety with an algebraic action of torus $(\C^\ast)^n$ with a dense orbit.
    A compact smooth toric variety is called a \emph{toric manifold}.
    One of the most important results for toric varieties, known as the \emph{fundamental theorem for toric geometry}, is that there is a bijection between the family of toric varieties and the family of \emph{fans}.
    In particular, each toric manifold corresponds to a complete non-singular fan.
    A complete non-singular fan can be regarded as a pair of a star-shaped simplicial sphere and the data of rays satisfying the non-singularity condition.
    Such a pair is called a fan-giving non-singular $\Z$-\emph{characteristic map} (simply, \emph{characteristic map}).
    Not only toric manifolds but also the classes of manifolds equipped with well-behaved torus actions (simply, \emph{toric spaces}) are also classified by the corresponding characteristic maps.
    The family of toric spaces includes toric manifolds and contains several categories such as quasitoric manifolds and topological toric manifolds.
    Furthermore, there are real analogues of toric spaces which are classified by $\Z_2$-characteristic maps.
    We refer to the introduction of \cite{CP13} for the summary of these classifications.

    It is natural to ask for an explicit classification of $\Z$- or $\Z_2$-characteristic maps in each category, although it is very complicated and, to date, only a few cases have been classified.
    Because the family of toric spaces is too large to handle, one reasonable approach would be to restrict our attention to a family of simplicial spheres with a few vertices.

    Let $K$ be a star-shaped simplicial sphere of dimension $n-1$ with $m$ vertices.
    The \emph{Picard number} $\Pic(K)$ of $K$ is defined by $\Pic(K):=m-n$.
    If $\Pic(K)=1$, then $K$ is the boundary complex of the $n$-simplex.
    It is known that $\CP^n$ is the only toric space supported by $K$.
    If $\Pic(K)=2$, then $K$ is the join of boundaries of two simplices (see \cite{Gru03}), and all characteristic maps over $K$ have been classified by Kleinschmidt \cite{K}.
    The classification of characteristic maps corresponding to toric manifolds over $K$ with $\Pic(K)=3$ is attributable to Batyrev \cite{Ba}, who used the fact that every toric manifold of Picard number $3$ is projective in his proof.
    Since a non-projective toric manifold of Picard number $4$ was constructed by Oda \cite{oda88}, Batyrev's method is not applicable to larger Picard numbers.

    Our previous research \cite{CP13} established a new way to classify characteristic maps in each category as follows.
    Let $K$  be a star-shaped simplicial sphere with $m$ vertices and fix a vertex $v$.
    Consider a 1-simplex $I$ whose vertices are labeled $v_{1}$ and $v_{2}$ and denote by $\partial I$ the 0-skeleton of $I$.
    Now, let us define a new simplicial complex on $m+1$ vertices, called the \emph{(simplicial) wedge} (or \emph{wedging}) of $K$ at $v$, denoted by $\wed_{v}(K)$, as
    \[ \wed_{v}(K)= (I \star \link_K\{v\}) \cup (\partial I \star (K\setminus\{v\})), \]
    where $K\setminus\{v\}$ is the induced subcomplex with $m-1$ vertices except $v$, the $\link_K\{v\}$ is the link of $v$ in $K$,
    and $\star$ is the join operation of simplicial complexes. {One notes that the resulting simplicial complex, denoted by $K(J)$, obtained from $K$ by a sequence of wedgings is determined by a positive integer tuple $J=(j_1, \ldots, j_m)$.
    Details will be given in Section~\ref{sec:wedge_seed}.}

    Let $(K,\lambda)$ be a characteristic map of dimension $n$ and $\sigma$ a face of $K$. % such that the vectors $\lambda(i)$, $i\in\sigma$, are unimodular.
    Then a characteristic map $(\link_K\sigma, \proj_\sigma \lambda)$, refered to as the \emph{projected characteristic map}, is defined by the following map
    \[
        (\proj_\sigma \lambda )(v) = [\lambda(v)]\in \Z^n/\langle \lambda(w)\mid w\in\sigma\rangle \cong \Z^{n - |\sigma|}.
    \]
    The authors \cite{CP13} showed that $\lambda$ is uniquely determined by the projections $\proj_{v_1}\lambda$ and $\proj_{v_2}\lambda$, see Proposition~\ref{prop:wedge1}.
    In other words, {roughly speaking}, characteristic maps on $K$ theoretically determine all possible characteristic maps on a wedge of $K$ in each category.
    We note that the wedge operation preserves the Picard number.
    Using this method, they succeeded in reproving Batyrev's classification without using the projectivity of toric manifolds.

    However, the following difficulties remain in terms of applying this method to general cases:
    \begin{enumerate}
      \item  A star-shaped simplicial sphere is called a \emph{seed} if it cannot be written as a simplicial wedge.
        In order to find all characteristic maps over $K$ with a fixed Picard number, we have to consider all seeds which support characteristic maps in each category.
        The number of seeds supporting characteristic maps should be sufficiently small for the method in \cite{CP13} to be practical.

      \item Even if we know every characteristic map over $K$, it is not easy to find all characteristic maps over {$\wed_v(K)$} because there exists a pair $(\lambda_1,\lambda_2)$ of characteristic maps over $K$ such that there is no $\lambda$ over $\wed_v(K)$ such that $\proj_{v_1} \lambda = \lambda_1$ and $\proj_{v_2} \lambda = \lambda_2$.
    Hence, to achieve our goal, i.e., {to classify all characteristic maps on $K(J)$ for any $J$}, we would have to determine which pair $(\lambda_1,\lambda_2)$ of characteristic maps over $K$ produces a characteristic map $\lambda$ over $\wed_v(K)$,
    and we would have to repeat this procedure for every stage in the sequence of wedgings.
    This would be very hard or practically impossible.
    \end{enumerate}

    The work we present in this paper aims to resolve the above two problems. Firstly, in Theorem~\ref{thm:finiteness_of_seeds}, we show that, for a fixed number $\ell$, there are at most finitely many seeds of Picard number $\ell$ supporting characteristic maps.
    Secondly, in Corollary~\ref{cor:main_theorem}, we investigate a systematic way to find all characteristic maps on $K(J)$ using \emph{puzzles} that only depend on a seed $K$. For each $J=(j_1, \ldots, j_m)$, an $R$-characteristic map over $K(J)$ corresponds to a color-preserving graph homomorphism, called a \emph{{realizable} puzzle}, from $G(J)$ to $D'(K)$ satisfying that the image of every square-shaped subgraph of $G(J)$ is {realizable}, where $R$ is either $\Z$ or $\Z_2$, $G(J)$ is the $1$-skeleton of a product of simplices $\prod_{i=1}^m \Delta^{j_i-1}$ with the specific coloring, and $D'(K)$ is the pre-diagram of $K$ which contains all $R$-characteristic maps over $K$ and the information of pairs $\{ \lambda_1, \lambda_2\}$ having $\lambda$ such that $\proj_{v_1} \lambda = \lambda_1$ and $\proj_{v_2} \lambda = \lambda_2$. The pre-diagram $D'(K)$ for $K$ equipped with the set of {realizable} squares is called the \emph{diagram} of $K$ and is denoted by $D(K)$. %The realizable squares in $D'(K)$ are determined only by $K$ and not by $J$
    {For given $K$ and $J$, we only need to find realizable puzzles from $G(J)$ to $D'(K)$ and this would not involve the repetitive tasks in (2).}
    The {precise} combinatorial interpretation of realizable puzzles will be given in Section~\ref{sec:example_of_puzzles}.
    These two facts present a practical way to classify the toric spaces of fixed Picard number.

    In addition, it is worthy to remark that Theorem~\ref{thm:finiteness_of_seeds} easily implies the main conjecture proposed by Batyrev \cite{Ba} whose proof can be found in Corollary~\ref{rem:proofofconj}. Let $\Sigma$ be a complete fan and let $G(\Sigma)$ be the set of all generators of rays of $\Sigma$.
    \begin{conjecture}[Conjecture~7.1 \cite{Ba}]\label{conj:batyrev}
        For any $n$-dimensional complete non-singular fan $\Sigma$ with Picard number $\ell$, there exists a constant $N(\ell)$ depending only on $\ell$ such that the number of primitive collections in $G(\Sigma)$ is always not more than $N(\ell)$.
    \end{conjecture}
    In his paper, Batyrev needed the projectivity for the classification of toric manifolds of Picard number 3. But his conjecture is for general case, as we did in this paper.

\section{Wedge operation and seed} \label{sec:wedge_seed}
    First of all, let us recall some notions about simplicial complexes. Recall that for a face $\sigma$ of a simplicial complex $K$, the \emph{link} of $\sigma$ in $K$ is the subcomplex
    \[
         \link_K\sigma := \{ \tau \in K \mid \sigma\cup\tau\in K,\;\sigma\cap\tau=\varnothing\}
    \]
    and the \emph{simplicial join} of two disjoint simplicial complexes $K_1$ and $K_2$ is defined by
    \[
        K_1 \star K_2 = \{ \sigma_1 \cup \sigma_2 \mid \sigma_1 \in K_1,\; \sigma_2 \in K_2\}.
    \]
    Let $K$ be a simplicial complex with vertex set $[m]=\{1,2,\ldots,m\}$ and fix a vertex $v$ in $K$. Consider a 1-simplex $I$ whose vertices are ${v_1}$ and ${v_2}$ and denote by $\partial I = \{v_1,\,v_2\}$ the 0-skeleton of $I$. Now, let us define a new simplicial complex on $m+1$ vertices, called the \emph{(simplicial) wedge} of $K$ at $v$, denoted by $\wed_{v}K$, by
    \[
        \wed_{v}K= (I \star \link_K\{v\}) \cup (\partial I \star (K\setminus\{v\})),
    \]
    where $K\setminus\{v\}$ is the induced subcomplex with $m-1$ vertices except $v$. See Figure~\ref{fig:wedge}.

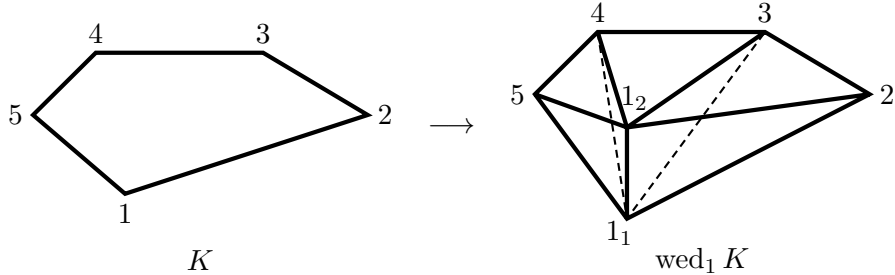
\begin{figure}[h]
    \begin{tikzpicture}[scale=.55]
        \coordinate [label=below:$1$](11) at (-9.8,0.6);
        \coordinate [label=right:$2$](22) at (-4,2.5);
        \coordinate [label=above:$3$](33) at (-6.5,4);
        \coordinate [label=above:$4$](44) at (-10.5,4);
        \coordinate [label=left:$5$](55) at (-12,2.5);
        \draw (-2,2.2) node {$\longrightarrow$};
        \coordinate [label={[xshift=-2.8pt,yshift=2.8pt]below:$1_1$}](1_1) at (2.2,0);
        \coordinate [label={[xshift=2.8pt,yshift=2.8pt]above:$1_2$}](1_2) at (2.2,2.2);
        \coordinate [label=right:$2$](2) at (8,3);
        \coordinate [label=above:$3$](3) at (5.5,4.5);
        \coordinate [label=above:$4$](4) at (1.5,4.5);
        \coordinate [label=left:$5$](5) at (0,3);
        \draw (-8,-1) node {$K$}
            (4,-1) node{$\wed_1K$};
        \draw [ultra thick] (11)--(22)--(33)--(44)--(55)--cycle
            (1_1)--(2)--(3)--(4)--(5)--cycle
            (1_2)--(1_1)
            (1_2)--(2)
            (1_2)--(3)
            (1_2)--(4)
            (1_2)--(5);
        \draw [thick, densely dashed] (4)--(1_1)--(3);
    \end{tikzpicture}
    \caption{Illustration of a wedge of $K$}\label{fig:wedge}
\end{figure}

    The operation itself is called the \emph{simplicial wedge operation} or \emph{(simplicial) wedging}.  Another description of simplicial wedging can be given using minimal non-faces. A subset $\sigma$ of the vertex set of $K$ is called a \emph{minimal non-face} of $K$ if $\sigma \notin K$ but every proper subset of $\sigma$ is a face of $K$. Note that every simplicial complex is determined by its minimal non-faces. Let $J=(j_1, \ldots, j_m)\in \Z_+^m$ be a vector of positive integers. Denote by $K(J)$ the simplicial complex on vertices
    \[
        \{ \underbrace{1_1,1_2,\ldots,1_{j_1}},\underbrace{{2_1},2_2,\ldots,{2_{j_2}}},\ldots, \underbrace{{m_1},\ldots,{m_{j_m}}} \}
    \]
    with minimal non-faces
    \[
        \{ \underbrace{{(i_1)_1},\ldots,{(i_1)_{j_{i_1}}}},\underbrace{{(i_2)_1},\ldots,{(i_2)_{j_{i_2}}}},\ldots, \underbrace{{(i_k)_1},\ldots,{(i_k)_{j_{i_k}}}} \}
    \]
    for each minimal non-face $\{{i_1},\ldots,{i_k}\}$ of $K$. It is an easy fact that $\wed_{i}K = K(J)$ where $J=(1,\ldots,1,2,1,\ldots,1)$ is the $m$-tuple with 2 as the $i$th entry.
    It should be noted that $K(J)$ can also be obtained from $K$ by a sequence of wedgings.
    See \cite{BBCG10} or \cite{CP13} for details.

%\red{One can see that $K(J)$ can be obtained by a sequence of wedgings from $K$. One can see that $\wed_{v_1}( \wed_v(K))$ is combinatorially equivalent to $\wed_{v_2}(\wed_v(K))$, and for two vertices $v$ and $w$ of $K$, $\wed_v(K)$ is combinatorially equivalent to $\wed_w (K)$. This implies that the result simplicial complex, denoted by $K(J)$, obtained from $K$ by a sequence of wedgings is only dependent on an integer tuple $J=(j_1, \ldots, j_m)$, where $j_i$ is the number of wedgings which choose vertices made from the $i$th vertex of $K$ in the process.}

    Let $R$ be the ring $\Z$ or $\Z_2=\Z/2\Z$. {Henceforth, let us assume that $K$ is an $(n-1)$-dimensional star-shaped simplicial sphere whose vertex set is $[m]$.} In this paper, we call a finite set of vectors $B=\{v_1,\dotsc, v_n\}\subset R^n$ an \emph{$R$-basis} or simply a \emph{basis} of $R^n$ if
    \begin{enumerate}
        \item $R=\Z$ and $B$ is unimodular in $\Z^n$,
        \item or $R=\Z_2$ and $B$ is {linearly independent in}
        %a basis of the vector space
        $\Z_2^n$.
    \end{enumerate}
    \begin{definition}
        A map $\lambda\colon [m] \to R^n$ is a (non-singular) $R$-\emph{characteristic map} over $K$, or simply a \emph{characteristic map}, if the following holds:
        \begin{equation}\label{eq:nonsingular}
            \text{if $\{i_1,\dotsc,i_n\}\in K$, then $\lambda(i_1),\dotsc,\lambda(i_n)$ forms an $R$-basis.}\tag{$\ast$}
        \end{equation}
        The condition \eqref{eq:nonsingular} is known as the \emph{non-singularity condition}.
    \end{definition}

    Every $\Z$-characteristic map $\lambda\colon [m]\to \Z^n$ induces a $\Z_2$-characteristic map given by the composition $[m] \to \Z^n \to \Z_2^n$, where the map $\Z^n \to \Z_2^n$ is the natural modulo map. This new characteristic map is frequently called the \emph{mod 2 reduction} of $\lambda$.

    According to Section~7.5 (more specifically, Corollary~7.33 and Proposition~7.34) of \cite{BP}, one concludes:
    \begin{theorem}\label{thm:lambdaphi}
        The following are equivalent.
        \begin{enumerate}
            \item $K$ admits an $R$-characteristic map.
            \item {There exists a map $\phi\colon [m] \to R^{m-n}$ such that for every maximal face $\{i_1,\dotsc,i_n\}\in K$, $\{\phi(i)\mid 1\le i \le m,\, i \ne i_k\text{ for } 1\le k \le n\}$ is an $R$-basis.}
        \end{enumerate}
    \end{theorem}

    Let us write $\Pic(K) := m-n$ and call it the \emph{Picard number of $K$}. Further, we fix a positive integer {$\ell$} and assume that $\Pic(K)=\ell$. We additionally assume that $K$ admits a $\Z_2$-characteristic map $\lambda$. If $m\ge 2^\ell$, the map $\phi$ of (2) of Theorem~\ref{thm:lambdaphi} cannot be one-to-one. Suppose that $v,w\in [m]$ and $\phi(v)=\phi(w)$. Then every facet of $K$ should contain either $v$ or $w$. We need {the following} lemma at this point.

    \begin{lemma}\label{lem:seed}
        Let $v$ and $w$ be distinct vertices of $K$. If every facet of $K$ contains either $v$ or $w$, then $K$ is isomorphic to $L \star \partial I$ or $\wed_v L$ for some simplicial complex $L$.
    \end{lemma}
    \begin{proof}
        Let $\sigma$ be a subset of the vertex set of $K$, $v\notin \sigma$ and $w \notin \sigma$. Then we claim that
        \[
            \sigma\in K \Leftrightarrow \sigma\cup\{v\} \in K \Leftrightarrow \sigma\cup\{w\} \in K.
        \]
        To prove the claim, suppose that $\sigma \in K$.
        {Then,}
         %By assumption,
         we have a facet $\tau$ of $K$ such that $\sigma\subseteq \tau$. If $\tau$ contains both $v$ and $w$, it is done. Otherwise, without loss of generality, we can assume that $v\in\tau$. Then $\tau \cup \{w\} \setminus\{v\}$ is also a facet of $K$ since $K$ is a pseudomanifold and there are exactly two facets of $K$ which are supersets of the $(n-2)$-simplex $\tau\setminus\{v\}$. Therefore the claim is proven. Considering the minimal non-faces of $K$, one easily concludes that $K$ is a suspension if {$\{v,w\} \not\in K$,} or $K$ is a wedge otherwise.
    \end{proof}
%    \red{
%    \begin{proposition}
%      Let $K$ be an $(n-1)$-dimensional star-shaped simplicial sphere, and $v$ a vertex of $K$. Then, $K$ admits an $R$-characteristic map if and only if so does $\wed_v(K)$.
%    \end{proposition}
%    }
%
%
%    Inspired by \red{Lemma~\ref{lem:seed}}, we define the notion of \emph{seeds}.
    \begin{definition}
        A star-shaped simplicial sphere is called a \emph{seed} if it cannot be written as a simplicial wedge.
    \end{definition}

    There are certainly a lot of seeds in the family of star-shaped simplicial spheres: for example, a flag simplicial complex cannot be a wedge. In spite of this, the following theorem says that ``good'' seeds are very rare in a sense of toric topology.

%    \begin{theorem}
%        For a fixed positive integer $\ell$, there are only finitely many seeds with Picard number $\ell$ which admit a characteristic map.
%    \end{theorem}
{
    \begin{theorem} \label{thm:finiteness_of_seeds}
        For a fixed positive integer $\ell$, if a seed $K$ of Picard number $\ell$ admits a characteristic map, then $m \leq 2^\ell - 1$, where $m$ is the number of vertices of $K$. As a corollary, there are only finitely many seeds with Picard number $\ell$ which admit a characteristic map.
    \end{theorem}
}
    \begin{proof}
{
        It is enough to prove the theorem for $R=\Z_2$. If $K$ admits a $\Z$-characteristic map, then its mod 2 reduction will work.
        We shall use an induction on $\ell$. It can easily be seen that the theorem holds for $\ell \leq 2$. Assume that it holds for a seed of Picard number $\ell-1$;  a seed of Picard number $\ell-1$ has at most $2^{\ell-1}-1$ vertices.
}

{
        Suppose $m \geq 2^\ell$. Since $K$ is a seed, by Lemma~\ref{lem:seed}, $K=L \ast \partial I$ for some simplicial complex $L$. We note that $L$ is also a seed. The number of vertices of $L$ is $m-1$, and $\Pic(L) = \Pic(K) -1 = \ell-1$.
        Since $m-1 \geq 2^\ell - 1 \geq 2^{\ell-1}$, by induction hypothesis, it is a contradiction.
        %First, there are only finitely many seeds such that $m < 2^{m-n}=2^\ell$. Thus suppose $m\ge 2^{m-n}$. Then by Lemma~\ref{lem:seed}, either $K=L \ast \partial I$ or $K=\wed_v L$ for some simplicial complex $L$. If $K=\wed_v L$, it is done since $K$ is not a seed. If $K=L\ast \partial I$, then $\Pic(L) = (m-2)-(n-1) = \Pic(K) - 1$. If $K=L\ast \partial I$ admits a characteristic map, then so is $L$. If $K$ is a seed, then $L$ is also a seed. Therefore, by induction on $\ell$, such $K$ has only finitely many possibilities.
}
    \end{proof}

    The above theorem easily gives an affirmative solution to Conjecture~\ref{conj:batyrev} originally proposed by Batyrev \cite{Ba} in 1991.
    \begin{corollary}\label{rem:proofofconj}
        For any $n$-dimensional complete non-singular fan $\Sigma$ with Picard number $\ell$, there exists a constant $N(\ell)$ depending only on $\ell$ such that the number of primitive collections in $G(\Sigma)$ is always not more than $N(\ell)$.
    \end{corollary}
    \begin{proof}
Observe that the primitive collections of a complete simplicial fan $\Sigma$ correspond one-to-one to the minimal non-faces of the underlying simplicial complex of $\Sigma$. Furthermore, the wedge operation does not change the number of minimal non-faces, reminding the definition of $K(J)$. By Theorem~\ref{thm:finiteness_of_seeds}, we have only finitely many seeds whose minimal non-faces to be counted. Thus an upper bound exists for given Picard number $\ell$.
     \end{proof}
\section{Puzzle and classification of toric spaces}
    In \cite{CP13} (especially in Section~4), the authors have studied classification of toric spaces over wedges of star-shaped simplicial spheres. In this section, we further improve Corollary~4.5 of \cite{CP13} and try to provide a combinatorial and systematic way to classify toric spaces over $K(J)$.

    A characteristic map $\lambda\colon [m] \to R^n$ can be represented by an $(n\times m)$-matrix
    \[
        \begin{pmatrix}
            \lambda(1) & \lambda(2) & \cdots & \lambda(m)
        \end{pmatrix},
    \]
    where $\lambda(1),\dotsc,\lambda(m)$ are column vectors. This matrix is called a \emph{characteristic matrix} and also frequently denoted by $\lambda$.
    \begin{definition}\label{def:D-J_equivalent}
        Two characteristic maps $\lambda_1,~\lambda_2 \colon [m] \to R^n$
        %{, or more generally two $(n\times m)$-matrices with primitive column vectors over $R$,}
        {are} said to be \emph{Davis-Januszkiewicz equivalent} or \emph{D-J equivalent} if they are the same up to change of basis of $R^n$. That is%({how about non-omnioriented case?})
        , $\lambda_1$ and $\lambda_2$ are D-J equivalent if one of their corresponding characteristic matrices can be changed to the other by finite applications of any of the following (called \emph{row operations}):
        \begin{enumerate}
            \item Multiply a row by $-1$.
            \item Add a multiple of one row to another row.
        \end{enumerate}
        The equivalence classes are called \emph{Davis-Januszkiewicz classes} or \emph{D-J classes}.
    \end{definition}
    \begin{remark}
        When $R=\Z$ or $R=\Z_2$, the above definition corresponds to a D-J equivalence of omnioriented topological toric manifolds or real topological toric manifolds, respectively. If $R= \Z_2$, then the row operation $(1)$ is redundant.
    \end{remark}

    {
    For a simplicial complex $K$, we denote by $V(K)$ the set of vertices of $K$.
    \begin{definition}
    Let us fix a face $\sigma \in K$. Let $\lambda\colon V(K) \to R^n$ be a map such that the vectors $\{\lambda(v)\}_{v\in\sigma}$ are unimodular. Then we define a map called the \emph{projection of $\lambda$ with respect to $\sigma$} by the following:
    %Then a characteristic map $(\link_K\sigma, \proj_\sigma \lambda)$, called the \emph{projected characteristic map}, is defined by the map
    \[
        (\proj_\sigma \lambda )(w) = [\lambda(w)]\in \Z^n/\langle \lambda(v)\mid v\in\sigma\rangle \cong \Z^{n - |\sigma|}
    \]
    for $w \in V(\link_{K}\sigma)$. The map $\proj_\sigma \lambda \colon \link_K\sigma \to \Z^{n-|\sigma|}$ is defined up to basis change of $\Z^{n-|\sigma|}$. When $(K,\lambda)$ is a characteristic map, then $(\link_K\sigma,\proj_\sigma \lambda)$ is also a characteristic map also called a \emph{projected characteristic map}. When $\sigma=\{v\}$ is a vertex, one can simply write $\proj_\sigma\lambda = \proj_v \lambda$.
    \end{definition}
    }

    {
    \begin{remark}\label{rem:projection}
    Fix a vertex $1$ of $K$ to be wedged at. To study toric spaces over $K(J)$, it is worth describing the projected characteristic map using matrices. The vertex set of its wedge $\wed_1K$ can be written as $\{1_1, 1_2, 2, \dotsc, m\}$. Let $\Lambda$ be a characteristic map over $\wed_1K$. After row operations, one can assume that $\Lambda(1_1)$ is a coordinate vector. In other words, the matrix $\Lambda$ can be written as
    \begin{equation}\label{eq:projection}
        \Lambda = \left(\begin{array}{c|ccc}
        1 &  a_1 & \cdots & a_m \\ \hline
        0 &      &        &  \\
        \vdots & &   \lambda    &  \\
        0 &      &        &
        \end{array}\right)_{(n+1)\times(m+1)}
    \end{equation}
    where the columns are labeled as $1_1, 1_2,2,\dotsc,m$. Then the characteristic matrix for $\proj_{1_1}\Lambda$ is the matrix $\lambda$, since the link $\link_{\wed_1K}\{1_1\}$ is naturally isomorphic to $K$.
    \end{remark}
    }

    {
    To define the pre-diagram, we choose one of the following categories of characteristic maps. Refer to \cite{CP13} for the definitions of positive orientedness and fan-givingness of characteristic maps.
    \begin{enumerate}
        \item The category of $\Z$-characteristic maps.
        \item The category of positively oriented $\Z$-characteristic maps.
        \item The category of fan-giving $\Z$-characteristic maps.
        \item The category of $\Z_2$-characteristic maps.
    \end{enumerate}
    In the next definition, we understand every characteristic map or D-J class to be an object in the chosen category. Note that each category corresponds to the following category of toric spaces:
    \begin{enumerate}
        \item The category of omnioriented topological toric manifolds.
        \item The category of almost complex topological toric manifolds.
        \item The category of toric manifolds.
        \item The category of real topological toric manifolds.
    \end{enumerate}
    }

    \begin{definition}\label{def:prediagram}
        The \emph{pre-diagram} of $K$, written as $D'(K)$, is an edge-colored pseudograph $(V,E)$ {satisfying}
        \begin{enumerate}
            \item the vertex set $V$ whose elements are the D-J classes over $K$, %{Definition of D-J classes. Here, we need omnioriented version since we only deal with row operations.}
            \item the edge set $E$ is defined as follows: $E$ is the collection of the sets $\{\lambda_1,\lambda_2,v\}$ where $\lambda_1,\lambda_2 \in V$ and $v\in V(K)$ such that there exists a characteristic map over $\wed_vK$ whose two projections onto $K$ are $\lambda_1$ and $\lambda_2$. The element $\{\lambda_1,\lambda_2,v\}$ is called an \emph{edge} colored $v$ whose endpoints are $\lambda_1$ and $\lambda_2$.
                %We say that the edge $e$ is \emph{colored} in $v$.
        \end{enumerate}
    \end{definition}
%    The pre-diagram $D'(K)$ is just an edge-colored graph with loops and multiple edges.
    \begin{remark}
        For every D-J class $\lambda$ over $K$ and a vertex $v$ of $K$, there is a loop of $D'(K)$ starting at $\lambda$ and colored $v$ because every loop indicates a canonical extension (see (4.2) of \cite{CP13}). Hence the loops are usually omitted when a drawing of the pre-diagram is made.
    \end{remark}

    {
    The pre-diagram is easily computable. Let $\lambda\colon [m] \to R^n$ be a characteristic map and fix a number $1 \le v \le m$. Then we obtain a matrix $\operatorname{Pr}(\lambda,v)$ such as in Remark~\ref{rem:projection}. To be more precise, after applying suitable row operations to $\lambda$, we can assume that the $v$th column is a coordinate vector.
    Then
    \begin{equation}
        \lambda = \left(\begin{array}{c|ccc}
        1 &  a_2 & \cdots & a_m \\ \hline
        0 &      &        &  \\
        \vdots & &   A    &  \\
        0 &      &        &
        \end{array}\right)_{n\times m }
    \end{equation}
    (when $v=1$ for example). Then we define $\operatorname{Pr}(\lambda,v) = A$. Note that $\operatorname{Pr}(\lambda,v)$ is defined up to row operations.
    \begin{lemma} \label{lemma:prediagram}
        For two D-J classes over $K$ and a vertex $v$ of $K$, the set $\{\lambda_1,\lambda_2,v\}$ is an edge of $D'(K)$ if and only if $\operatorname{Pr}(\lambda_1,v)$ and $\operatorname{Pr}(\lambda_2,v)$ are the same after row operations.
        %the same map up to the choice of basis of $R^{n-1}$.
    \end{lemma}
%    We obtain a projection-like map $\operatorname{Pr}(\lambda,v) \colon [m]\setminus\{v\} \to R^n/\langle \lambda(v)\rangle$ defined as
%    \[
%        \operatorname{Pr}(\lambda,v)(i) = [\lambda(i)] \in R^n/\langle \lambda(v)\rangle \cong R^{n-1}.
%    \]
%    We remark that $\operatorname{Pr}(\lambda,v)$ also can be expressed by an $n \times (m-1)$-matrix
%    $$
%        \begin{pmatrix}
%            [\lambda(1)] & \cdots & [\lambda(v-1)] & [\lambda(v+1)] & \cdots & [\lambda(m)]
%        \end{pmatrix}
%    $$ while this matrix is denoted by $\operatorname{Pr}(\lambda,v)$ again.
%    Similarly to characteristic matrices, two projection-like maps $\operatorname{Pr}(\lambda_1,v)$ and $\operatorname{Pr}(\lambda_2,v)$ are said to be \emph{equivalent} if one is obtainable from the other by finite applications of row operations in Definition~\ref{def:D-J_equivalent}.
    \begin{proof}
        In \eqref{eq:projection}, one can further assume that $\Lambda(1_2)$ is a coordinate vector, since $\{1_1,1_2\}$ is a face of $\wed_1K$ and thus $\{\Lambda(1_1),\Lambda(1_2)\}$ is a unimodular set. Therefore, after row operations, one obtains
        \begin{equation}
            \Lambda = \left(\begin{array}{cc|ccc}
            1 & 0 &  a_2 & \cdots & a_m \\
            0 & 1 &  b_2 & \cdots & b_m \\ \hline
            0 & 0 &    &        &  \\
            \vdots & \vdots & &  A    &  \\
            0 & 0  &  &        &
            \end{array}\right)_{(n+1)\times(m+1)}
        \end{equation}
        and $\operatorname{Pr}(\lambda_1,v) = \operatorname{Pr}(\lambda_2,v) = A$. The converse is similar.
    \end{proof}
    }

    Let $J=(j_1, \ldots, j_m) \in \Z_+^m$ be an $m$-tuple whose coordinates are positive integers.
    We consider a colored graph $G(J)$ with $m$ colors constructed as follows: $G = G(J)$ is the graph determined by the $1$-skeleton of the simple polytope $\Delta^{j_1-1} \times \Delta^{j_2-1} \times \cdots \times \Delta^{j_m-1}$, where $\Delta^j$ is the $j$-dimensional simplex.
%        We shall give a color to each edge.}
%        Pick an edge $e$ of $G$. Then $e$
    One remarks that each edge $e$ of $G$ can be uniquely written as
    \[
        p_1 \times p_2 \times \dotsb \times p_{v-1}\times e_v \times p_{v+1} \times \dotsb \times p_m,
    \]
    where $p_i$ is a vertex of $\Delta^{j_i-1}$, $1\le i \le m$, $i\ne v$, and $e_v$ is an edge of $\Delta^{j_v-1}$.
    %If the vertex set of $K$ is $[m]$, then
    {Then} we color $v\in[m]$ on the edge $e$.

    \begin{definition}
        A graph homomorphism $p\colon G(J) \to D'(K)$ which preserves the edge coloring is called a \emph{puzzle} of the pair $(K,J)$.
        %We say that a puzzle is \emph{reducible} if its image contains an edge of $D'(K)$ which is a loop. The corresponding edge of $G(J)$ is called \emph{trivial} in the puzzle, that is,
        An edge $e = \{\balpha,\balpha'\}$ of $G(J)$ is called \emph{trivial} in a puzzle $p$ if $p(\balpha) = p(\balpha')$.
        A puzzle is said to be \emph{reducible} if it contains a trivial edge, or \emph{irreducible} otherwise.
        %When $J = (1,\dotsb,1,3,1,\dotsb,1)$ or $J=(1,\dotsb,1,2,1,\dotsb,1,2,1,\dotsb,1)$, $G(J)$ is a cycle graph $C_3$ or $C_4$ and a puzzlev$G(J) \to D'(K)$ is called a \emph{\red{triangle}} or a \emph{square} respectively. Here, $C_i$ means the cycle graph with $i$ vertices.
    \end{definition}
    \begin{definition}
        Let $p\colon G(J) \to D'(K)$ be a puzzle. If a subgraph $G'$ of $G(J)$ corresponds to a face of $\Delta^{j_1-1} \times \Delta^{j_2-1} \times \cdots \times \Delta^{j_m-1}$, the graph homomorphism $p|_{G'}\colon G' \to D'(K)$ is called a \emph{subpuzzle} of $p$. If $G'$ is the 1-skeleton of a hypercube $(\Delta^1)^d$, then the corresponding subpuzzle is called a \emph{subcube}. In particular, when $d=2$, it is also called a \emph{subsquare}. We remark that a subpuzzle can be regarded as a puzzle so that $G' \cong G(J')$ for some $m$-tuple $J'$.
    \end{definition}
%    By definition, an edge $e = \{v_1, v_2\}$ of $G(J)$ is trivial in a puzzle $p$ if and only if $p(v_1) = p(v_2)$.

    For $J = (j_1, \dotsc, j_m) \in \Z_+^m$, let us denote by $I(J)$ the index set
    \[
        I(J) = \{\balpha = (\alpha_1,\dotsc,\alpha_m)\in \Z_+^m \mid 1\le \alpha_i\le j_i \text{ for }1\le i \le m\}.
    \]
    For an index $\balpha=(\alpha_1,\dotsc,\alpha_m)\in I(J)$, we consider a face of the simplicial complex $K(J)$
    \[
        \sigma(\balpha):= \{1_1,1_2,\dotsc,1_{j_1},
        \,\dotsc,\,
        m_1,m_2,\dotsc,m_{j_m}\}\setminus \{1_{\alpha_1},2_{\alpha_2},\dotsc,m_{\alpha_m} \}.
    \]
    Observe that $\sigma(\balpha)$ is indeed a face since it does not contain any minimal non-faces of $K(J)$. Furthermore, the complex $\link_{K(J)}\sigma(\balpha)$ is naturally isomorphic to $K$.
    Denote by $V(G(J))$ the vertex set of the graph $G(J)$.
    Then, there is an obvious identification $I(J) \to V(G(J))$ by labeling the vertices of $\Delta^{j_i-1}$ by $i_1,\dotsc,i_{j_i}$, $1\le j\le m$.  Using this identification, for any characteristic map $\lambda$ over $K(J)$, we obtain a puzzle of $(K,J)$ by assigning to the vertex $\alpha$ the characteristic map $\proj_{\sigma(\balpha)}\lambda$. Puzzles constructed in this way are said to be \emph{realizable}.
    \begin{remark}\label{rem:faceofrealizablepuzzle}
        It is easy to verify that a realizable puzzle is indeed a puzzle. More precisely, let $p\colon G(J) \to D'(K)$ be a realizable puzzle defined by $\lambda$. Then, it is necessary to show that every edge of $G(J)$ maps to an edge of $D'(K)$. For any subgraph $G'$ corresponding to a face of $\Delta^{j_1-1} \times \Delta^{j_2-1} \times \cdots \times \Delta^{j_m-1}$, it can be shown that the map $p|_{G'} \colon G' \to D'(K)$ is a realizable puzzle defined by a projection of $\lambda$. In particular, when $G'$ is an edge of $G(J)$, this argument guarantees that the image of an edge of $G(J)$ is indeed an edge of $D'(K)$, thereby completing the proof.
    \end{remark}
    For example, when $J = (2,3,1,\dotsc,1)$, then a realizable puzzle $p\colon G(J) \to D'(K)$ will look like
    
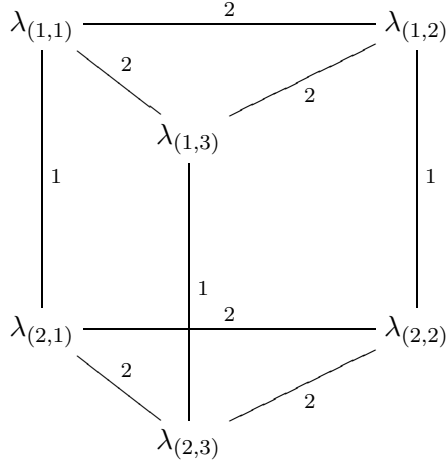
\[
        \xymatrix{
            \lambda_{(1,1)} \ar@{-}[rrr]^2 \ar@{-}[rd]^2 \ar@{-}[ddd]^1 & & & \lambda_{(1,2)} \ar@{-}[ddd]^1 \ar@{-}[dll]^2\\
             & \lambda_{(1,3)} \ar@{-}[ddd]^1 & & \\
             & & & \\
            \lambda_{(2,1)} \ar@{-}[rrr]^2 \ar@{-}[dr]^2 & & & \lambda_{(2,2)}\ar@{-}[dll]^2 \\
             & \lambda_{(2,3)} & &
        }
    \]
    where $\lambda_{\balpha}  = \proj_{\sigma(\balpha)}\lambda$ with the subscripts $\balpha$ abbreviated so that $(\alpha_1,\alpha_2) = (\alpha_1,\alpha_2,1,1,\dotsc,1)$.
    Conversely, every realizable puzzle of $(K,J)$ determines a unique characteristic map over $K(J)$ up to Davis-Januszkiewicz equivalence by {the} uniqueness of {the} characteristic maps over wedges (see Proposition~\ref{prop:wedge1}). In conclusion, we have a bijection
    \begin{align*}
        \{\text{D-J classes over $K(J)$}\} &\to \{\text{realizable puzzles of $(K,J)$}\}. \\
        \Lambda\qquad \qquad & \mapsto \qquad\qquad p(\Lambda)
    \end{align*}
    Therefore, for the classification of toric spaces over $K(J)$, the main challenge is to determine whether a given puzzle is realizable or not.
    For example, the puzzle
    
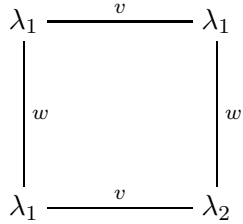
\[
        \xymatrix{
                \lambda_{1} \ar@{-}[rr]^v \ar@{-}[dd]^w  & & \lambda_{1} \ar@{-}[dd]^w \\
                  &   &  \\
                \lambda_{1}\ar@{-}[rr]^v &   & \lambda_{2} }
    \]
    is not realizable whenever $\lambda_1 \ne \lambda_2$ (see Proposition~\ref{prop:standardform}).

\section{Criterion on the realizability of puzzles}
    To determine whether a given puzzle $p$ is realizable or not, let us firstly consider subcubes of $p$.

    \begin{theorem}\label{thm:subcube}
        A puzzle $G(J) \to D'(K)$ is realizable if and only if all of its subcubes are realizable.
    \end{theorem}
    Before proving the above theorem, we first start from an edge. Let
    \[
        \xymatrix{
            \lambda_1 \ar@{-}[rr] ^v & & \lambda_2
        }
    \]
    be an edge $e$ of the {pre-diagram $D'(K)$}. This edge can be regarded as a realizable puzzle corresponding to $\lambda$ over $\wed_v K$ whose two projections are $\lambda_1$ and $\lambda_2$.
    \begin{lemma}\label{lem:edge}
        In the setting above, assume that the characteristic matrix for $\lambda_1$ is
        \[
            \lambda_1 = \begin{pmatrix}
                \mathbf{a}_1 & \mathbf{a}_2 & \cdots & \mathbf{a}_m
            \end{pmatrix}_{n\times m}
        \]
        where $\a_i$ denotes the $i$th column vector of $\lambda_1$. Then the D-J class $\lambda$ can be expressed by a matrix of the form
        \begin{equation}\label{eqn:edgestandard}
            \lambda = \begin{pmatrix}
                1    & \cdots & v-1      &  v_1 & v_2& v+1 & \cdots & m    \\ \hline
                \a_1 & \cdots & \a_{v-1} & \a_v & 0  & \a_{v+1} & \cdots & \a_m \\
                e_1  & \cdots & e_{v-1}  & -1   &  1 &  e_{v+1} & \cdots & e_m
            \end{pmatrix}_{(n+1)\times (m+1)},
        \end{equation}
        for appropriate integers $e_i$ for $i\ne v$. The numbers above the horizontal line are indicators for vertices of simplicial complexes. If every $e_i$ is zero, then the characteristic map is a canonical extension.
    \end{lemma}
    \begin{proof}
        The proof is easy once the basis of $\Z^{n+1}$ is appropriately selected. More precisely, we select $\lambda(v_2)$ such that it is a coordinate vector. Since $\proj_{v_2} \lambda = \lambda_1$, $\lambda$ must be of the form
        \[
            \lambda = \begin{pmatrix}
                \a_1 & \cdots & \a_{v-1} & \a_v & 0  & \a_{v+1} & \cdots & \a_m \\
                \ast  & \cdots & \ast  & x   &  1 &  \ast & \cdots & \ast
            \end{pmatrix}_{(n+1)\times (m+1)},
        \]
        recalling that $\proj_{v_2} \lambda$ is obtained from $\lambda$ by deleting the column $v_2$ and the $(n+1)$th row. As $\a_v$ is a primitive vector, by adding to the last row a suitable linear combination of the first $n$ rows, it can be assumed that $x=-1$.
    \end{proof}
    The matrix \eqref{eqn:edgestandard} is called a \emph{standard form} for the edge $e$ centered at $\lambda_1$. If every $e_i$ is zero, then $\lambda_1 = \lambda_2$ and $\lambda$ is simply a canonical extension of $\lambda_1$.
    In this case, the edge $e$ is a loop in $D'(K)$. Note that two standard forms (of the same center) for the same edge can be changed to each other by using row operations. The standard form is unique in this sense.

    Let $p\colon G(J) \to D'(K)$ be a puzzle. Recall that the vertex set $V(G(J))$ is identified with the index set $I(J)$ and thus $p(\balpha)$ is a D-J class over $K$ for all $\balpha = (\alpha_1, \dotsc, \alpha_m)\in I(J)$. We fix an element $\1 = (1,1,\dotsc,1)\in I(J)$ and let $p(\1)$ be the characteristic matrix
    \[
        p(\1) = \begin{pmatrix}
            \a_1 & \cdots & \a_m
        \end{pmatrix}.
    \]
    In the graph $G(J)$, suppose that $\1$ and $\balpha$ are connected by an edge colored $v$. Then $\alpha_i = 1$ for $i\ne v$. We label the edge connecting $\1$ and $\balpha$ as $e_{v}^{\alpha_v}$. For each edge $e_{v}^{\alpha_v}$, we fix a standard form of $e_{v}^{\alpha_v}$ centered at $p(\1)$ like the following:
    \[
            \begin{pmatrix}
                1    & \cdots & v-1      &  v_1 & v_2& v+1 & \cdots & m    \\ \hline
                \a_1 & \cdots & \a_{v-1} & \a_v & 0  & \a_{v+1} & \cdots & \a_m \\
                e_{v,1}^{\alpha_v}  & \cdots & e_{v,v-1}^{\alpha_v}  & -1   &  1 &  e_{v,v+1}^{\alpha_v} & \cdots & e_{v,m}^{\alpha_v}
            \end{pmatrix}_{(n+1)\times (m+1)}.
    \]

    For convenience of notation of block matrices, we use the following notations:
    \[
        A_i = \begin{pmatrix}
            \a_i & 0 & \cdots & 0
        \end{pmatrix}_{n\times j_i},
    \]
    \[
        S_i = \begin{pmatrix}
            -1       & 1 &        & 0 \\
            \vdots   &   & \ddots &   \\
            -1       & 0 &        & 1
        \end{pmatrix}_{(j_i-1)\times j_i},
    \]
    and
    \[
        e_{k,i} = \begin{pmatrix}
            e_{k,i}^2 & 0      & \cdots & 0 \\
            e_{k,i}^3 & 0      & \cdots & 0 \\
            \vdots       & \vdots & \ddots & \vdots \\
            e_{k,i}^{j_k} & 0    & \cdots & 0
        \end{pmatrix}_{(j_k-1)\times j_i}.
    \]

    \begin{proposition}\label{prop:standardform}
        If the puzzle $p\colon G(J) \to D'(K)$ is realizable, the following matrix
        \begin{equation}\label{eqn:standardform}
            \Lambda = \begin{pmatrix}
                A_1     &  A_2     & A_3     &  \cdots  & A_m     \\
                S_1     &  e_{1,2} & e_{1,3} &  \cdots  & e_{1,m} \\
                e_{2,1} & S_2      & e_{2,3} &  \cdots  & e_{2,m} \\
                \vdots  &          & \ddots  &          & \vdots  \\
                e_{m-1,1}& \cdots  &         &  S_{m-1} & e_{m-1,m}  \\
                e_{m,1} &  \cdots  &         &  e_{m,m-1} & S_m
            \end{pmatrix}
        \end{equation}
        defines a characteristic map $\Lambda$ over $K(J)$ where the columns are labeled as
        \[
            1_1,\dotsc,1_{j_1},2_1,\dotsc,2_{j_2},\dotsc,m_1,\dotsc,m_{j_m}
        \]
        from left to right. In particular, the realizable puzzle is uniquely determined by $p(\1)$ and $p(\balpha)$ for vertices $\balpha$ adjacent to $\1$.
    \end{proposition}
    We call \eqref{eqn:standardform} a \emph{standard form} of the characteristic map $\Lambda$ centered at $\1$.
%    The following proposition is an important observation for reducible realizable puzzles and is very useful for actual computation.
%    \begin{proposition}\label{prop:reduciblecube}
%        Let $p\colon G(J) \to D'(K)$ be a realizable reducible cube. That is, there are two vertices
%        \[
%            \balpha = (\alpha_1,\dotsc,\alpha_{i-1},1,\alpha_{i+1},\dotsc,\alpha_m)
%        \]
%        and
%         \[
%            \bbeta = (\alpha_1,\dotsc,\alpha_{i-1},2,\alpha_{i+1},\dotsc,\alpha_m)
%        \]
%        of the cube consisting an edge of the cube such that $p(\balpha) = p(\bbeta)$. Then, for any two vertices of the form
%        \[
%            \balpha' = (\alpha'_1,\dotsc,\alpha'_{i-1},1,\alpha'_{i+1},\dotsc,\alpha'_m)
%        \]
%        and
%         \[
%            \bbeta' = (\alpha'_1,\dotsc,\alpha'_{i-1},2,\alpha'_{i+1},\dotsc,\alpha'_m),
%        \]
%        $p(\balpha') = p(\bbeta')$.
%    \end{proposition}
%    \begin{proof}
%        By the above remark, the cube is a canonical extension and the two projections corresponding to the canonical extension are the same.
%    \end{proof}

    \begin{proof}
        The proof is once again conducted by a choice of basis which similar to Lemma~\ref{lem:edge}. Note that Lemma~\ref{lem:edge} deals with a single wedge and we can obtain $K(J)$ by consecutive wedge operations from $K$. Although we present the construction for specific $J=(2,3,1,\dotsc,1)$, the general case is identical. Let us start from the characteristic map for $K(2,1,\dotsc,1)=\wed_1 K$
        \[
            \begin{pmatrix}
                   \a_1    & 0 &   \a_2    &   \a_3    &   \cdots  \\
                   -1      & 1 &  \ast    &   \ast    &   \cdots
            \end{pmatrix}.
        \]
        Next, we perform a wedge operation at the vertex 2 and obtain
        \[
            \begin{pmatrix}
                   \a_1 &0       & \a_2  &0  &   \a_3    &   \cdots  \\
                   -1   &1       & \ast  &0  &   \ast    &   \cdots  \\
                   \ast &x       & y     &1  &   \ast    &   \cdots
            \end{pmatrix}.
        \]
        We add $-x$ times the $(n+1)$th row to the $(n+2)$th row. After that, we can set $y= -1$ by adding to the $(n+2)$th row a linear combination of the first $n$ rows. The result is as follows:
        \[
            \begin{pmatrix}
                  \a_1 &0   &   \a_2 &0   &   \a_3    &   \cdots  \\
                  -1   &1   &   \ast &0   &   \ast    &   \cdots  \\
                  \ast &0   &   -1   &1   &   \ast    &   \cdots
            \end{pmatrix}.
        \]
        We apply a wedge at the vertex 2 (regardless of whether it is $2_1$ or $2_2$) to obtain
        \[
            \begin{pmatrix}
                  \a_1 &0   &   \a_2 &0    &0    &   \a_3    &   \cdots  \\
                  -1   &1   &   \ast &0    &0    &   \ast    &   \cdots  \\
                  \ast &0   &   -1   &1    &0    &   \ast    &   \cdots  \\
                  \ast &a   &\  b    &c    &1    &   \ast    &   \cdots
            \end{pmatrix}.
        \]
        Our aim is $a=0$, $b=-1$, and $c=0$. For this, add $-a$ times the $(n+1)$th row and $-c$ times the $(n+2)$th row to the $(n+3)$th row. As before, a linear combination of the first $n$ rows is added to the $(n+3)$th row to obtain $b=-1$.
        After this construction, we see that every characteristic map should have the form of \eqref{eqn:standardform} except the numbers $e_{k,i}^r$. Let us return to
        \begin{equation}\label{eqn:J23}
            \Lambda = \begin{pmatrix}
                  \a_1      &0 & \a_2      &0 &0 & \a_3     & \cdots  \\
                  -1        &1 & f_{1,2}^2 &0 &0 & f_{1,3}^2& \cdots  \\
                  f_{2,1}^2 &0 & -1        &1 &0 & f_{2,3}^2& \cdots  \\
                  f_{2,1}^3 &0 & -1        &0 &1 & f_{2,3}^3& \cdots
            \end{pmatrix}
        \end{equation}
        for some $f_{k,i}^r \in \Z$.
        Observe that, for example,
         \[
            \proj_{1_2}(\proj_{2_3}\Lambda)= \begin{pmatrix}
                  \a_1       & \a_2      &0 & \a_3     & \a_4 & \cdots  \\
                  f_{2,1}^2  & -1        &1 & f_{2,3}^2& f_{2,4}^2 & \cdots
            \end{pmatrix}
        \]
        is the standard form for the edge $e_2^2$ in the realizable puzzle. Therefore,  by {the} uniqueness of standard forms for edges, we can assume that $f_{k,i}^r = e_{k,i}^r$ after further row operations. It is easy to observe that this argument generally holds. Recall that
        \[
            \sigma(\balpha):= \{1_1,1_2,\dotsc,1_{j_1},
        \,\dotsc,\,
        m_1,m_2,\dotsc,m_{j_m}\}\setminus \{1_{\alpha_1},2_{\alpha_2},\dotsc,m_{\alpha_m} \}.
        \]
        Similarly, for an edge $e$ whose endpoints are $\balpha$ and $\bbeta$, define the simplex
        \[
        \sigma(e):= \sigma(\balpha) \cap \sigma(\bbeta).
        \]
        Note that $\link_{K(J)}\sigma(e_v^{\alpha_v})$ is naturally isomorphic to $\wed_v K$ for any $v$ and $\alpha_v$. In addition, note that
        \[
            \sigma(e_v^{\alpha_v})= \{1_1,1_2,\dotsc,1_{j_1},
        \,\dotsc,\,
        m_1,m_2,\dotsc,m_{j_m}\}\setminus \{1_1,2_1,\dotsc,m_1,v_{\alpha_v} \}
        \]
        and $\proj_{\sigma(e_v^{\alpha_v})} \Lambda$ has the desired standard form for every $v$ and $\alpha_v$.
    \end{proof}
 Two edges
 \begin{align*}
   e &= p_1 \times \dotsb \times p_{v-1}\times e_v \times p_{v+1} \times \dotsb \times p_m \quad \text{ and } \\
   e' &= p'_1 \times \dotsb \times p'_{v-1}\times e'_v \times p'_{v+1} \times \dotsb \times p'_m
 \end{align*}of $G(J)$  that are colored $v$  are said to be \emph{parallel} if $e_v = e'_v$.

\begin{corollary}
  A realizable puzzle corresponds to a canonical extension if and only if the puzzle is reducible. Furthermore, if an edge $e$ of a realizable puzzle is trivial, then the edges parallel to $e$ are also trivial.
\end{corollary}
\begin{proof}
    If a trivial edge $e$ exists, we may assume that $e$ is colored $v$ and its endpoints are $\1$ and $\balpha = (\alpha_1,\dotsc, \alpha_m)$ so that $p(\balpha) = p(\1)$. Then we have observed that the edge corresponds to a canonical extension and every $e_i$ is zero in the standard form $\Lambda$ in \eqref{eqn:edgestandard} for the edge. That means there is a row of the matrix of \eqref{eqn:standardform} such that its $v_1$th entry is $-1$, $v_{\alpha_v}$th entry is $1$ and the remaining entries are zero ($e_{v,i}^{\alpha_v}=0$ for all $i$). Therefore, by Lemma~\ref{lem:edge}, it immediately follows that $\Lambda$ also represents a canonical extension such that
    \begin{equation}\label{eqn:trivial}
        \proj_{v_1} \Lambda = \proj_{v_{\alpha_v}} \Lambda.
    \end{equation}
    Let $e'$ be an edge parallel to $e$ whose endpoints are $\bbeta$ and $\bbeta'$. Then there is a face $\sigma$ of $K(J)$ such that $\sigma\cup\{v_1\}$ and $\sigma\cup\{v_2\}$ are also faces of $K(J)$ and
    \[
        p(\bbeta)=\proj_{\sigma\cup\{v_1\}} \Lambda \quad\text{ and } \quad p(\bbeta')=\proj_{\sigma\cup\{v_2\}} \Lambda.
    \]
    Therefore they are equal.
\end{proof}
%
%    \begin{remark}
%      Suppose that $\1$ and $\balpha = (\alpha_1,\dotsc, \alpha_m)$ are connected by an edge in $G(J)$ and $p(\balpha) = p(\1)$. Then we have observed that the edge corresponds to a canonical extension and every $e_i$ is zero in \eqref{eqn:edgestandard}. Therefore, it is immediate that the matrix \eqref{eqn:standardform} also represents a canonical extension. In conclusion,
%    \end{remark}
%
%    When one considers the characteristic map, it usually satisfy the non-singularity condition. Instead of this, one can think of the notion of \emph{generalized characteristic maps}.
%    \begin{definition}
%        For a given $(n-1)$-dimensional star-shaped simplicial sphere $K$, a map $\lambda\colon V(K) \to R^n$ is called a \emph{generalized characteristic map} if the vector $\lambda(v)$ is a non-zero primitive vector for all $v\in V(K)$.
%    \end{definition}
%    We remark that if a generalized characteristic map $\lambda$ is non-singular, that is, $\lambda(i_1), \ldots, \lambda(i_n)$ forms an $R$-basis for $\{i_1, \ldots, i_n\}\in K$, then $\lambda$ is indeed an $R$-characteristic map. We further remark that if a generalized characteristic map $\lambda$ is simplicial, that is, $\lambda(i_1), \ldots, \lambda(i_n)$ are linearly independent over $R$, then it corresponds to a primitive complete characteristic map in \cite{CP13}.
%

    For later use, we state a version of a result in \cite{CP13}.
    \begin{proposition}[Proposition~4.4, \cite{CP13}]\label{prop:wedge1}
        Let $\lambda\colon V(\wed_vK) \to R^{n+1}$ be a map which is not necessarily a characteristic map. Label the two new vertices as $v_1$ and $v_2$. Moreover assume that the set $\{\lambda(v_1),\lambda(v_2)\}$ is unimodular. Then $\lambda$ is a characteristic map if and only if $\proj_{v_1}\lambda$ and $\proj_{v_2}\lambda$ are characteristic maps over $K$.
    \end{proposition}

    Now we are ready to prove Theorem~\ref{thm:subcube}.

    \begin{proof}[Proof of Theorem~\ref{thm:subcube}]
        One direction is straightforward by Remark~\ref{rem:faceofrealizablepuzzle}. To prove the converse, assume that every subcube of $p\colon G(J) \to D'(K)$ is realizable and assume that $j_v \ge 2$ for all $v$ (if not, some canonical extensions may be used). At this time we are not sure whether $p$ is realizable, but we can still follow the construction of Proposition~\ref{prop:standardform} and obtain the matrix $\Lambda$ in \eqref{eqn:standardform}. We recursively apply Proposition~\ref{prop:wedge1} to show that $\Lambda$ is non-singular. Observe that in \eqref{eqn:standardform}, $\{\Lambda(v_2),\,\Lambda(v_3),\dotsc,\, \Lambda(v_{j_v})\}$ is a unimodular set for any $v$ such that $1 \le v \le m$ because all vectors are coordinate vectors. Apply Proposition~\ref{prop:wedge1} to two vertices $v_a$ and $v_b$ of $K(J)$ other than $v_1$. By recalling that the projection simply involves the deletion of a column and a row, we obtain two ``smaller'' standard forms. Continuing this process, we reach matrices of the following form
        \[
            \begin{pmatrix}
                \a_1    &   0   &   \a_2    & 0   & \a_3 & \cdots & \a_m & 0 \\
                -1      &   1   &   \ast    & 0   & \ast &     & \ast& 0 \\
                \ast    &   0   &   -1      & 1   & \ast &     &  \vdots   & 0 \\
                \vdots  & \vdots& & & \ddots & &  &\\
                \ast    &  0    &    \ast   & 0   &    &     & -1  & 1
            \end{pmatrix}.
        \]
        Here, we terminate our process since the available pair of coordinate vectors is not found. However, this matrix is simply a standard form of a subcube, which is realizable by assumption. In particular, it is non-singular. Finally, a recursive application of Proposition~\ref{prop:wedge1} shows that $\Lambda$ is non-singular, proving that the puzzle $p$ is realizable.
    \end{proof}

\section{Squares to Cubes}

%    Let $K$ be an $(n-1)$-dimensional star-shaped simplicial sphere with $m$ vertices and let $\Lambda$ be a Davis-Januszkiewicz class over $K$. The D-J class can be represented by a matrix
%    \[
%    \Lambda = \begin{pmatrix}
%        \a_1 & \a_2 & \dotsc & \a_m
%    \end{pmatrix}_{n\times m}.
%    \]
%    We denote by $\row \Lambda$ the sublattice of $\Z^m$ spanned by the row vectors of $\Lambda$ and call it the \emph{row space} of $\Lambda$. Note that $\row \Lambda$ does not depend on the choice of the characteristic matrix.
%
%    Consider a sublattice $L$ of $\Z^m$ and pick a face $\sigma = \{i_1,\dots,i_k\}$ of $K$. Then a natural projection map $f_\sigma \colon L \to \spa\{v_{i_1},\dotsc,v_{i_k}\}\cong \Z^k$ is defined, where $v_i$ is the $i$-th coordinate vector of $\Z^m$. We call $L$ \emph{$K$-unimodular} if $f_\sigma$ is surjective for any $\sigma \in K$. We call $L$ \emph{column-primitive} if $f_\sigma$ is surjective for any vertex $\sigma$ of $K$.
%
%    An immediate bijection arises:
%    \begin{align*}
%        \{\text{D-J class over $K$}\} &\longleftrightarrow \{\text{rank $n$ $K$-unimodular sublattice}\} \\
%        \Lambda &\longleftrightarrow  \row\Lambda.
%    \end{align*}
%    We intentionally confuse a D-J class over $K$, and a rank $n$ $K$-unimodular sublattice. In general, every $L$ be a rank $k$ sublattice of $\Z^m$ can be written by a $(k\times m)$-matrix whose row vectors are the generators of $L$.
    \begin{definition}
        A \emph{marked row vector} in $\Z^m$ is a pair $(\r,v)$ where $1\le v \le m$ and $\r = (r_1,\dotsc,r_m) \in \Z^m$ is a row vector. The number $v$ is called the \emph{marking} of the marked row vector. When there is no danger of confusion, the markings can simply be omitted to write $(\r,v) = \r$. When $r_v = 0$, the marked row vector is called \emph{reduced}.
    \end{definition}
    Let us fix a characteristic map $\Lambda$ over $K$ and let $(\r,v)$ and $(\r',v)$ be two marked row vectors in $\Z^m$ with the same markings. We write $\r\sim_\Lambda \r'$ if $\r-\r'$ is a linear combination of row vectors of $\Lambda$ over $\Z$. Assume that $\r = (r_1,\dotsc,r_m) \in \Z^m$ and
    \[
        \Lambda = \begin{pmatrix}
            \a_1 & \dotsc & \a_m
        \end{pmatrix}_{n\times m}.
    \]
    Let us use the notation
    \[
        \begin{pmatrix}
            \Lambda \\
            (\r,v)
        \end{pmatrix} :=
        \begin{pmatrix}
            1    & \cdots & v-1      &  v_1    & v_2 & v+1      & \cdots & m   \\ \hline
            \a_1 & \cdots & \a_{v-1} &  \a_{v} & 0   & \a_{v+1} & \cdots & \a_m \\
            r_1  & \cdots & r_{v-1}  & -1+r_v  & 1   & r_{v+1}  & \cdots & r_m
        \end{pmatrix}_{(n+1)\times (m+1)}.
    \]
    Note that if the matrix $\begin{pmatrix} \Lambda \\ \r \end{pmatrix}$ defines a characteristic map for $\wed_v K$, it is simply a standard form if $\r$ is reduced. Observe that
    \begin{center}
    $\begin{pmatrix} \Lambda \\ (\r,v) \end{pmatrix}$ are D-J equivalent to $\begin{pmatrix} \Lambda \\ (\r',v) \end{pmatrix}$ if and only if
            $\r\sim_\Lambda \r'$.
    \end{center}

    Consider its two projections over $K$. The projection $\proj_{v_2}\begin{pmatrix} \Lambda \\ \r \end{pmatrix}$ is simply $\Lambda$, whereas the other projection is denoted by $\Lambda^\r := \proj_{v_1}\begin{pmatrix} \Lambda \\ \r \end{pmatrix}$.

    Let us calculate $\Lambda^\r$ explicitly. Without loss of generality, we assume that the marking of $\r$ is $1$. Then by definition,
    \[
        \begin{pmatrix} \Lambda \\ \r \end{pmatrix} =
        \begin{pmatrix}
               \a_1    & 0 &   \a_2    &   \a_3    &   \cdots & \a_m  \\
               -1+r_1  & 1 &  r_2      &   r_3     &   \cdots & r_m
        \end{pmatrix}.
    \]
    In fact, one may further assume that $\r$ is reduced, that is, $r_1 = 0$ (see the proof of Lemma~\ref{lem:edge}). After suitable row operations, it becomes
    \[
        \begin{pmatrix}
               0       & \a_1 &   \a_2 + r_2\a_1    &   \a_3 + r_3\a_1    &   \cdots & \a_m+ r_m\a_1  \\
               -1      & 1    &  r_2                &      r_3            &   \cdots & r_m
        \end{pmatrix}.
    \]
    Deleting the $1_1$th column and the last row, we obtain the projection
    \begin{equation}\label{eqn:lambdar}
        \Lambda^\r = \begin{pmatrix}
            \a_1 &   \a_2 + r_2\a_1    &   \a_3 + r_3\a_1    &   \cdots & \a_m+ r_m\a_1
        \end{pmatrix}.
    \end{equation}

    The notation $\begin{pmatrix} \Lambda \\ \r \end{pmatrix}$ can be extended in an obvious way. Let $\r^1,\,\dotsc,\,\r^k$ be marked row vectors with mutually distinct markings. By relabeling the vertices of $K$, we may assume that $\r^v = (\r^v,v)$ has marking $v$ for $v=1,\dotsc,k$. Then we define
    \[
        \begin{pmatrix} \Lambda \\ \r^1 \\ \vdots \\ \r^k \end{pmatrix}
        := \begin{pmatrix}
          \a_1      & 0 &     \a_2  & 0 & \cdots & \a_k  & 0 &  \a_{k+1}  & \cdots & \a_m \\
          -1+r^1_1  & 1 &     r^1_2 & 0 & \cdots & r^1_k & 0 &  r^1_{k+1} & \cdots & r^1_m \\
          r^2_1     & 0 & -1+r^2_2  & 1 & \cdots & r^2_k & 0 &  r^2_{k+1} & \cdots & r^2_m \\
          \vdots    &   &           &   & \ddots & \vdots&   &  \vdots    & \ddots &       \\
          r^k_1     & 0 &    r^k_2  & 0 & \cdots & -1+r^k_k&1& r^k_{k+1}  & \cdots & r^k_m
        \end{pmatrix}.
    \]
    Again, if the above matrix defines a characteristic map over $K(J)$, where the first $k$ coordinates of $J = (2,\dots,2,1,\dotsc,1)$ are $2$ and the remaining coordinates are $1$, the matrix is a standard form if $\r^1,\dotsc,\r^k$ are reduced.

    \begin{lemma}\label{lem:cubefacet}
        In the above settings, we have the following identity:
        \[
            \proj_{k_1}\begin{pmatrix} \Lambda \\ \r^1 \\ \vdots \\ \r^k \end{pmatrix} =
            \begin{pmatrix} \Lambda^{\r^k} \\ (\r^1)^{\r^k} \\ \vdots \\ (\r^{k-1})^{\r^k} \end{pmatrix},
        \]
        where $\r^\s = (\r,v)^{(\s,w)}$ is a marked row vector in which marking $v$ has the defining property
        \[
            \proj_{w_1}\begin{pmatrix} \Lambda \\ \r \\ \s \end{pmatrix} = \begin{pmatrix}\Lambda^\s \\ \r^\s      \end{pmatrix}.
        \]
    \end{lemma}
        To prove the lemma, it is necessary to show that $\r^\s$ is well-defined up to the equivalence relation $\sim_\Lambda$. In the matrix $\begin{pmatrix} \Lambda \\ \r \\ \s \end{pmatrix}$, one can assume that $\r=(\r,1)$ and $\s=(\s,2)$ and that $\r$ and $\s$ are reduced as before. Then
        \[  \begin{pmatrix} \Lambda \\ \r \\ \s \end{pmatrix} =
            \begin{pmatrix}
                \a_1  &  0   &  \a_2  &  0    & \a_3 &  \cdots  &  \a_m \\
                -1    &  1   &  r_2     &  0  & r_3  &  \cdots  &   r_m \\
                s_1     &  0 &  -1    &  1    & s_3  &  \cdots  &   s_m
            \end{pmatrix}.
        \]
        After a series of row operations, the matrix becomes
        \[
            \begin{pmatrix}
            \a_1+s_1\a_2 & 0 & 0 & \a_2 & \a_3+s_3\a_2 & \cdots & \a_m + s_m\a_2 \\
            -1+s_1r_2    & 1 & 0 & r_2  & r_3+s_3r_2   & \cdots & r_m + s_mr_2 \\
                s_1     &  0 &  -1    &  1    & s_3  &  \cdots  &   s_m
            \end{pmatrix},
        \]
        after which the $2_1$th column and the last row are deleted to obtain
        \[
            \proj_{2_1}\begin{pmatrix} \Lambda \\ \r \\ \s \end{pmatrix} =
            \begin{pmatrix}
            \a_1+s_1\a_2 & 0  & \a_2 & \a_3+s_3\a_2 & \cdots & \a_m + s_m\a_2 \\
            -1+s_1r_2    & 1  & r_2  & r_3+s_3r_2   & \cdots & r_m + s_mr_2
            \end{pmatrix}.
        \]
        From \eqref{eqn:lambdar} and the above formula, one concludes that
        \[
            \r^\s = \begin{pmatrix}
                s_1r_2   &   r_2  &  r_3 + s_3r_2   & r_4 + s_4r_2 \cdots  &   r_m + s_mr_2
            \end{pmatrix}.
        \]
        Care should be taken that $\r^\s$ itself is not reduced if $s_1r_2 \ne 0$. If we want $\r^\s$ to be reduced, then we can select a row vector $\r'$ such that $\r'\sim_\Lambda \r^s$.
    \begin{proof}[Proof of Lemma~\ref{lem:cubefacet}]
        The proof is a straightforward calculation once one assumes that $\r^1,\dotsc,\r^k$ are reduced.
    \end{proof}

%
%        By Proposition~\ref{prop:standardform}, a realizable cube $p$ corresponds to a characteristic matrix of standard form
%        \[
%            \begin{pmatrix} \Lambda \\ r^1 \\ \vdots \\ r^k \end{pmatrix}
%        \]
%        with reduced marked row vectors $r^1,\dotsc,r^k$. The standard form is centered at $\1$ and $p(\1) = \Lambda$. The cube has $2k$ facets. Among them, $k$ facets contain the vertex $\1$ and the other $k$ facets do not.

    Now we are ready to prove the following.
    \begin{lemma}\label{thm:squaretocube}
        A cube is realizable if and only if all of its subsquares are realizable.
    \end{lemma}
    \begin{proof}
        In this proof, we assume that $J = (2,\dots,2,1,\dotsc,1)$ has $k$ 2's in its coordinates. The set $I(J)$ of vertices of the corresponding cube is the same as $\{1,2\}^k$. To simplify the notation, we omit the last $(m-k)$ 1's of $\balpha = (\alpha_1,\dotsc,\alpha_k,1,\dots,1) \in I(J)$ and write $\balpha = (\alpha_1,\dotsc,\alpha_k)$.

        For the ``only if'' part, we have a stronger fact that every subpuzzle of a realizable puzzle is again realizable (see Remark~\ref{rem:faceofrealizablepuzzle}).
        For the ``if'' part, we use induction on the dimension $k$ of the cube. When $k=2$, it is a square and thus we are done. For $v=1,\dotsc,k$, let us denote $\h_v = (1,\dotsc,1,2,1,\dotsc,1) \in I(J)$ the vertex of the $k$-dimensional cube whose $v$th coordinate is $2$.
        There are exactly $k$ edges connected to the center vertex $\1$. Recall that every edge of a puzzle produces a characteristic map over a wedge of $K$. Therefore if $p(\1)=\Lambda$, then one can assume that the edge connecting $\1$ and $\h_v$ corresponds to the standard form $\begin{pmatrix} \Lambda \\ \r^v \end{pmatrix}$ where $\r^v$ is a reduced marked row vector with marking $v$ for $v=1,\dotsc,k$. By uniqueness of standard forms, the proof is complete if it is shown that the matrix
        \[
            \phi := \begin{pmatrix} \Lambda \\ \r^1 \\ \vdots \\ \r^k \end{pmatrix}
        \]
        provides a realization of the cube, i.e., it represents a characteristic map. By induction hypothesis, the facet $p_1$ determined by $\1, \h_1,\dotsc,\h_{k-1}$ and its opposite facet $p_2$ are realizable. Let us denote by $\h_{v,k}$ the fourth vertex of the square determined by $\1, \h_v, \h_k$. Then $p_2$ is the cube determined by $\h_k,\h_{1,k},\h_{2,k},\dotsc,\h_{k-1,k}$. We study the edges connecting $\h_{k}$ and $\h_{v,k}$. The square determined by $\1,\h_v, \h_k, \h_{v,k}$ is realizable by the hypothesis and its characteristic map is
        \[
            \phi_v := \begin{pmatrix} \Lambda \\ \r^v \\  \r^k \end{pmatrix}.
        \]
        Then
        \[
            \proj_{k_1}\phi_v = \begin{pmatrix} \Lambda^{\r^k} \\ (\r^v)^{\r^k} \end{pmatrix}
        \]
        is the characteristic map corresponding to the edge connecting $\h_{k}$ and $\h_{v,k}$. Thus the standard form for $p_2$ is
        \[
            \begin{pmatrix} \Lambda^{\r^k} \\ (\r^1)^{\r^k} \\ \vdots \\ (\r^{k-1})^{\r^k} \end{pmatrix},
        \]
        which is exactly $\proj_{k_1}\phi$ by Lemma~\ref{lem:cubefacet}. As we already know, the standard form for $p_1$ is
        \[
            \begin{pmatrix} \Lambda  \\  \r^1  \\ \vdots \\  \r^{k-1} \end{pmatrix}
        \]
        and it is exactly $\proj_{k_2}\phi$. Because the set of column vectors $\{\phi(k_1),\phi(k_2)\}$ is unimodular, we can apply Proposition~\ref{prop:wedge1} to show that $\phi$ is non-singular.
    \end{proof}

    By combining Theorem~\ref{thm:subcube} and Lemma~\ref{thm:squaretocube}, we obtain the main result:
    \begin{theorem}\label{thm:main}
        A puzzle is realizable if and only if all of its subsquares are realizable.
    \end{theorem}

    \begin{definition}
        A subgraph of $D'(K)$ which is the image of a realizable subsquare is called a \emph{realizable square} of $D'(K)$.
        {The \emph{diagram} for $K$, denoted by $D(K)$, is the pre-diagram $D'(K)$ for $K$ equipped with the set of realizable squares of $D'(K)$.} As Definition~\ref{def:prediagram}, a diagram $D(K)$ depends on the choice of category.
    \end{definition}
    \begin{corollary} \label{cor:main_theorem}
        The diagram $D(K)$ provides a complete classification of toric spaces over $K(J)$ for any $J$ in each category.
    \end{corollary}
    \begin{remark}\label{rem:reducible}
        Note that the square of the form
        \[
            \xymatrix{
                    \lambda_{1} \ar@{-}[rr]^v \ar@{-}[dd]^w  & & \lambda_{2} \ar@{-}[dd]^w \\
                      &   &  \\
                    \lambda_{1}\ar@{-}[rr]^v &   & \lambda_{2} }
        \]
        is always realizable because it corresponds to a canonical extension of the edge
        \[
            \xymatrix{
                \lambda_1 \ar@{-}[rr] ^v & & \lambda_2.
            }
        \]
        In general, a reducible puzzle is realizable if and only if
        \begin{enumerate}
            \item every edge parallel to a trivial edge is trivial, and
            \item a maximal irreducible subpuzzle is realizable.
        \end{enumerate}
    \end{remark}

\section{Combinatorics of puzzles} \label{sec:example_of_puzzles}
In this section, we discuss the combinatorial interpretation of realizable puzzles and provide a few examples.
Let $K$ be a star-shaped simplicial sphere on $[m]$ and $J = (j_1, \ldots, j_m) \in \Z_+^m$ an $m$-tuple of positive integers. A realizable puzzle $p \colon G(J) \to D'(K)$ can be interpreted as a way to fill a board with stones by obeying the following rules:
\begin{itemize}
  \item \textbf{Stones} : all characteristic maps $\lambda_1, \ldots, \lambda_\ell$ up to D-J classes over $K$.
  \item \textbf{Board} $G(J)$ : a stone is placed on each vertex of $G(J)$.
  \item \textbf{Rules} $D(K)$ : a diagram $D(K)$ is a pre-diagram $D'(K)$ equipped with the list of realizable squares.
    \begin{enumerate}
      \item Two stones $\lambda_i$ and $\lambda_j$ can be connected on $G(J)$ by an edge colored $v$ only when $\lambda_i$ and $\lambda_j$ are connected in $D'(K)$ by an edge colored $v$.
      \item Any square on $G(J)$ of form
    \[
        \xymatrix{
                \lambda_{1} \ar@{-}[rr]^v \ar@{-}[dd]^w  & & \lambda_{2} \ar@{-}[dd]^w \\
                  &   &  \\
                \lambda_{3}\ar@{-}[rr]^v &   & \lambda_{4} }
    \] is realizable.
    \end{enumerate}
\end{itemize}

Now, let us provide a few examples of puzzles for $R = \Z_2$ which correspond to real topological toric manifolds over $K$. The reader can find other cases in \cite{CP}, for example.

\begin{example}
  Let $K$ be the boundary complex of a cyclic $4$-polytope with $7$ vertices. It is easy to see that $K$ is a seed. Fixing an order of vertices, it is known that there are two $\Z_2$-characteristic maps up to D-J equivalence
    \[
        \lambda_1=\begin{pmatrix}
            1&0&0&0&1&0&1\\
            0&1&0&0&0&1&1\\
            0&0&1&0&1&1&0\\
            0&0&0&1&1&1&1
        \end{pmatrix}\text{  and  }
        \lambda_2=\begin{pmatrix}
            1&0&0&0&1&1&1\\
            0&1&0&0&1&0&1\\
            0&0&1&0&0&1&1\\
            0&0&0&1&1&1&0
        \end{pmatrix}.
    \] It is clear that there is no $\lambda$ on a wedge of $K$ whose two projections are $\lambda_1$ and $\lambda_2$ respectively, so its pre-diagram $D'(K)$ is a discrete graph with two vertices.
    Since it has no edge, for any $J \in \Z_+^7$, each puzzle only contains stones of one particular kind. In conclusion, there are only two puzzles for any $J$.
\end{example}

\begin{example}
Let $K =  \{1,2,3,4,5,6, 12,23,34,45,56, 61\}$ be the boundary complex of a hexagon on $[6]=\{1,\dotsc,6\}$. Then, there are eleven $\Z_2$-characteristic maps, and they can be separated into four types up to their rotational symmetries.

\begin{itemize}
  \item \textbf{TYPE 1 : ababab type} $$\lambda_1 = \begin{pmatrix} 1 & 0 & 1 & 0 & 1 & 0 \\ 0 & 1 & 0 & 1 & 0 & 1 \end{pmatrix}$$
  \item \textbf{TYPE 2 : abcbcb type}
  \begin{align*}
  \lambda_{2,1} = \begin{pmatrix} 1 & 0 & 1 & 0 & 1 & 0 \\ 0 & 1 & 1 & 1 & 1 & 1 \end{pmatrix} \quad
  \lambda_{2,2} = \begin{pmatrix} 1 & 0 & 1 & 1 & 1 & 1 \\ 0 & 1 & 0 & 1 & 0 & 1 \end{pmatrix} \\
  \lambda_{2,3} = \begin{pmatrix} 1 & 0 & 1 & 0 & 1 & 0 \\ 0 & 1 & 1 & 1 & 0 & 1 \end{pmatrix} \quad
  \lambda_{2,4} = \begin{pmatrix} 1 & 0 & 1 & 1 & 1 & 0 \\ 0 & 1 & 0 & 1 & 0 & 1 \end{pmatrix} \\
  \lambda_{2,5} = \begin{pmatrix} 1 & 0 & 1 & 0 & 1 & 0 \\ 0 & 1 & 0 & 1 & 1 & 1 \end{pmatrix} \quad
  \lambda_{2,6} = \begin{pmatrix} 1 & 0 & 1 & 0 & 1 & 1 \\ 0 & 1 & 0 & 1 & 0 & 1 \end{pmatrix}
  \end{align*}
  \item \textbf{TYPE 3 : abcacb type}
  \begin{align*}
  \lambda_{3,1} = \begin{pmatrix} 1 & 0 & 1 & 1 & 1 & 0 \\ 0 & 1 & 1 & 0 & 1 & 1 \end{pmatrix} \quad
  \lambda_{3,2} = \begin{pmatrix} 1 & 0 & 1 & 1 & 0 & 1 \\ 0 & 1 & 0 & 1 & 1 & 1 \end{pmatrix} \\
  \lambda_{3,3} = \begin{pmatrix} 1 & 0 & 1 & 0 & 1 & 1 \\ 0 & 1 & 1 & 1 & 0 & 1 \end{pmatrix}
  \end{align*}
  \item \textbf{TYPE 4 : abcabc type}
  $$
  \lambda_{4} = \begin{pmatrix} 1 & 0 & 1 & 1 & 0 & 1 \\ 0 & 1 & 1 & 0 & 1 & 1 \end{pmatrix}
  $$
\end{itemize}

By Lemma~\ref{lemma:prediagram}, one can obtain the following pre-diagram $D'(K)$
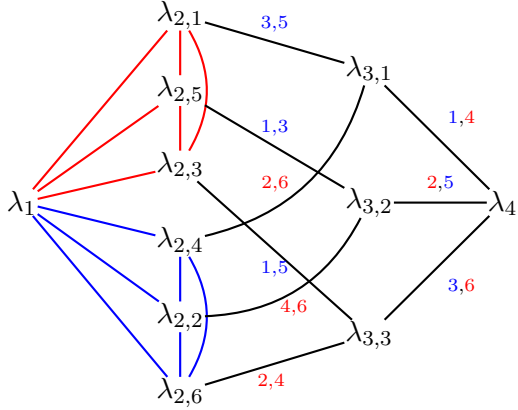
\begin{center}
\tikzstyle{v}=[ inner sep=0pt, minimum width=4pt]
\begin{tikzpicture}[thick,scale=0.5]
    \node [v] (bot32) {$\lambda_{3,2}$};
    \node [v] (bot31) [above of=bot32, node distance=50pt]{$\lambda_{3,1}$} ;
    \node [v] (bot33) [below of=bot32, node distance=50pt]{$\lambda_{3,3}$} ;
    \node [v] (bot4) [right of=bot32, node distance=50pt]{$\lambda_{4}$} ;
    \node [v] (bot21) [above left of=bot32, node distance=100pt]{$\lambda_{2,1}$} ;
    \node [v] (bot25) [below of=bot21]{$\lambda_{2,5}$} ;
    \node [v] (bot23) [below of=bot25]{$\lambda_{2,3}$} ;
    \node [v] (bot24) [below of=bot23]{$\lambda_{2,4}$} ;
    \node [v] (bot22) [below of=bot24]{$\lambda_{2,2}$} ;
    \node [v] (bot26) [below of=bot22]{$\lambda_{2,6}$} ;
    \node [v] (bot1) [left of=bot32, node distance=130pt]{$\lambda_{1}$} ;
    \path
    (bot1) edge [color=red] (bot21)
           edge [color=red] (bot23)
           edge [color=red] (bot25)
           edge [color=blue] (bot22)
           edge [color=blue] (bot24)
           edge [color=blue] (bot26)
    (bot21) edge [color=red] (bot25)
            edge [color=red, bend left] (bot23)
    (bot25) edge [color=red] (bot23)
    (bot24) edge [color=blue] (bot22)
            edge [color=blue, bend left] (bot26)
    (bot22) edge [color=blue] (bot26)
    (bot31) edge node[above] {\blue{\tiny 3,5}} (bot21)
            edge [bend left] node[left]{\red{\tiny 2,6 }} (bot24)
    (bot32) edge [bend left] node[below]{\red{\tiny 4,6 }} (bot22)
            edge node[above] {\blue{\tiny 1,3}} (bot25)
    (bot33) edge node[below] {\blue{\tiny 1,5}} (bot23)
            edge node[below]{\red{\tiny 2,4 }} (bot26)
    (bot4) edge node[above right] {\tiny \blue{1},\red{4}} (bot31)
           edge node[above]{\tiny \red{2},\blue{5}} (bot32)
           edge node[below right ]{\tiny \blue{3},\red{6}} (bot33);
\end{tikzpicture}
\end{center}
where
\begin{itemize}
  \item a red edge represents triple edges each of which is colored $2$,$4$, and $6$ respectively;
  \item a blue edge represents triple edges each of which is colored $1$,$3$, and $5$ respectively; and
  \item an edge with $\{i,j\}$ represents double edges each of which is colored $i$ and $j$ respectively.
\end{itemize}

One can check that the all possible irreducible rectangles in $D'(K)$ are realizable. The following are some examples:
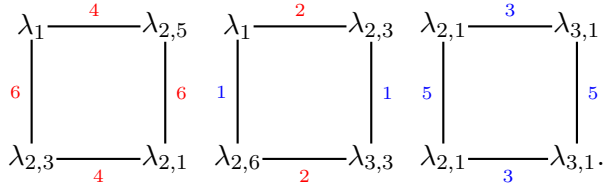
\begin{center}
\tikzstyle{v}=[ inner sep=0pt, minimum width=4pt]
\begin{tikzpicture}[thick,scale=0.5]
    \node [v] (bot1) {$\lambda_{1}$};
    \node [v] (bot21) [right of=bot1, node distance=50pt]{$\lambda_{2,5}$};
    \node [v] (bot31) [below of=bot21, node distance=50pt]{$\lambda_{2,1}$} ;
    \node [v] (bot24) [below of=bot1, node distance=50pt]{$\lambda_{2,3}$} ;
    \path
    (bot1) edge node[above] {\red{\tiny 4}} (bot21)
           edge node[left] {\red{\tiny 6}} (bot24)
    (bot31) edge node[right] {\red{\tiny 6}} (bot21)
            edge node[below] {\red{\tiny 4}} (bot24);
\end{tikzpicture}
\begin{tikzpicture}[thick,scale=0.5]
    \node [v] (bot1) {$\lambda_{1}$};
    \node [v] (bot21) [right of=bot1, node distance=50pt]{$\lambda_{2,3}$};
    \node [v] (bot31) [below of=bot21, node distance=50pt]{$\lambda_{3,3}$};
    \node [v] (bot24) [below of=bot1, node distance=50pt]{$\lambda_{2,6}$} ;
    \path
    (bot1) edge node[above] {\red{\tiny 2}} (bot21)
           edge node[left] {\blue{\tiny 1}} (bot24)
    (bot31) edge node[right] {\blue{\tiny 1}} (bot21)
            edge node[below] {\red{\tiny 2}} (bot24);
\end{tikzpicture}
\begin{tikzpicture}[thick,scale=0.5]
    \node [v] (bot1) {$\lambda_{2,1}$};
    \node [v] (bot21) [right of=bot1, node distance=50pt]{$\lambda_{3,1}$};
    \node [v] (bot31) [below of=bot21, node distance=50pt]{$\lambda_{3,1}.$};
    \node [v] (bot24) [below of=bot1, node distance=50pt]{$\lambda_{2,1}$} ;
    \path
    (bot1) edge node[above] {\blue{\tiny 3}} (bot21)
           edge node[left] {\blue{\tiny 5}}  (bot24)
    (bot31) edge node[right] {\blue{\tiny 5}} (bot21)
            edge node[below] {\blue{\tiny 3}} (bot24);
\end{tikzpicture}
\end{center}
    The following is an example of realizable puzzles on $G(2,3,1,1,1,1)$.
    
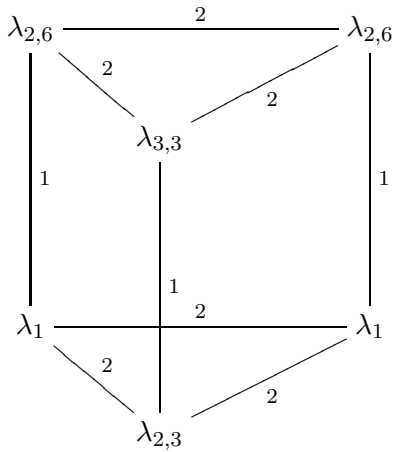
\[
        \xymatrix{
            \lambda_{2,6} \ar@{-}[rrr]^2 \ar@{-}[rd]^2 \ar@{-}[ddd]^1 & & & \lambda_{2,6} \ar@{-}[ddd]^1 \ar@{-}[dll]^2\\
             & \lambda_{3,3} \ar@{-}[ddd]^1 & & \\
             & & & \\
            \lambda_{1} \ar@{-}[rrr]^2 \ar@{-}[dr]^2 & & & \lambda_{1}\ar@{-}[dll]^2 \\
             & \lambda_{2,3} & &
        }
    \]
    Indeed, the puzzle has two irreducible squares and one reducible square. Since all irreducible squares  are on the list of realizable squares and a reducible square is automatically realizable by Remark~\ref{rem:reducible}, the above puzzle is a realizable puzzle by Theorem~\ref{thm:main}.
    Furthermore, one can check that there are $119$ realizable puzzles on $G(2,3,1,1,1,1)$. We leave it as an exercise.

%First, we observe that edges with at most $3$ different colors can be non-trivial.
%
%\textbf{CASE 1 : there is no non-trivial edge.} There are $11$ vertices in $D'(K)$. Hence, we have $11$ mod $2$ puzzles.
%
%\textbf{CASE 2 : there is at least one non-trivial edge.}
%
%\textbf{Subcase 2-1 : all non-trivial edges are colored $2, 4$, and $6$.} We consider the case that the $p(v)$ and $p(u_{i,j})$ are in $\{ \lambda_1, \lambda_{2,1}, \lambda_{2,3}, \lambda_{2,5}\}$, where $i=2,4,6$ and $j=1, \ldots, a_i$. In this case, there are $4^{a_2+a_4+a_6+1}-1$ choices. Furthermore, each choice determines one puzzle. Otherwise, all irreducible squares comes from edges $\lambda_{2,4}-\lambda_{3,1}$, $\lambda_{2,2}-\lambda_{3,2}$ and $\lambda_{2,6}-\lambda_{3,3}$. In this case, there are $2^{a_2 + a_4 +1} -1 + 2^{a_4 + a_6 +1} -1 + 2^{a_6 + a_2 +1} -1$ mod $2$ puzzles.
%
%\textbf{Subcase 2-2 : all non-trivial edges are colored $1, 3$, and $5$.} Similarly, we observe that the $p(v)$ and $p(u_{i,j})$ are in $\{ \lambda_1, \lambda_{2,2}, \lambda_{2,4}, \lambda_{2,6}\}$, where $i=1,3,5$ and $j=1, \ldots, a_i$. Therefore, there are $4^{a_1+a_3+a_5+1}-1$ choices, and each choice determines one puzzle. Otherwise, there are $2^{a_1 + a_3 +1} -1 + 2^{a_3 + a_5 +1} -1 + 2^{a_5 + a_1 +1} -1$ mod $2$ puzzles.
%
%\textbf{Subcase 2-3 : otherwise.} In this case, only the last three non-trivial square can appear. So there are
%      $$
%        (2^{a_2+a_6+1}-2)(2^{a_3+a_5+1}-2) + (2^{a_4+a_6+1}-2)(2^{a_1+a_3+1}-2) + (2^{a_2+a_4+1}-2)(2^{a_1+a_5+1}-2)
%      $$
%      mod $2$ puzzles.
\end{example}

We remark that there are finitely many realizable puzzles for given $K$ when $R = \Z_2$; therefore, it would be interesting to enumerate the number of realizable puzzles over $K(J)$.

\begin{question}
  Count the realizable puzzles over $K(J)$ where $R=\Z_2$ and $K$ is an $n$-gonal simplicial complex.
\end{question}


\bigskip

\begin{thebibliography}{Abc}

\bibitem{BBCG10}
A.~Bahri, M.~Bendersky, F.~R.~Cohen, and S.~Gitler, \emph{Operations on polyhedral products and a new topological construction of infinite families of toric manifolds}, \texttt{arXiv:1011.0094} (2010).

\bibitem{Ba}
V.~V.~Batyrev, \emph{On the classification of smooth projective toric varieties}, Tohoku Math. J. (2) 43 (1991), no.~4, 569--585.


\bibitem{BP}
V.~M. Buchstaber and T.~E. Panov, \emph{Torus actions and their
  applications in topology and combinatorics}, University Lecture Series,
  vol.~24, American Mathematical Society, Providence, RI, 2002.

\bibitem{CP13}
S.~Choi and H.~Park, \emph{Wedge operations and torus symmetries}, to appear in Tohoku Math. J., \texttt{http://arxiv.org/abs/1305.0136} (2013).

\bibitem{CP}
S.~Choi and H.~Park, \emph{Wedge operations and a new family of projective toric manifolds}, preprint (2015).


\bibitem{Gru03}
B.~Gr\"{u}nbaum, \emph{Convex polytopes}, second ed., Graduate Texts in Mathematics,
vol. 221, Springer-Verlag, New York, 2003.

\bibitem{K}
P. Kleinschmidt, \emph{A classification of toric varieties with few generators}, Aequationes Math. 35 (1988), no.~2-3, 254--266.

\bibitem{oda88}
T.~Oda, \emph{Convex Bodies and Algebraic Geometry.
An Introduction to the Theory of Toric Varieties},
Ergeb. Math. Grenzgeb. (3), 15, Springer-Verlag, Berlin, 1988.


\end{thebibliography}
\end{document}